\documentclass[preprint,12pt]{elsarticle}

\usepackage{amssymb}
\usepackage{amsthm,amsmath}

\theoremstyle{plain}
\newtheorem{thm}{Theorem}[section]
\newtheorem{prop}[thm]{Proposition}
\newtheorem{lem}[thm]{Lemma}
\newtheorem{cor}[thm]{Corollary}

\theoremstyle{definition}

\newtheorem{exple}[thm]{Example}
\newtheorem{nota}[thm]{Notation}

\theoremstyle{remark}
\newtheorem{rmk}[thm]{Remark}

\usepackage[english,francais]{babel}
\usepackage[latin1]{inputenc}
\usepackage[colorlinks]{hyperref}

\newcommand{\cro}{\{\:,\:\}}
\newcommand{\prs}{\langle\;,\,\rangle}
\newcommand{\anch}{{\pi_\sharp}}
\newcommand{\coor}{(x_1,\ldots,x_d)}

\newcommand{\efd}{\Omega^{*}(M)}
\newcommand{\ef}{\Omega^{1}(M)}

\newcommand{\ecv}{\mathfrak{X}^{1}(M)}
\newcommand{\efl}{\mathcal{C}^{\infty}(M)}
\newcommand{\cinf}{\mathcal{C}^{\infty}}

\def\to{\rightarrow}

\def\D{{\cal D}}
\def\F{{\cal F}}
\def\g{{\frak g}}
\def\H{{\cal H}}

\def\Li{{\cal L}}
\def\M{{\cal M}}
\def\S{{\cal S}}
\def\T{{\bf T}}

\def\R{{\mathbb R}}

\def\ad{{\rm ad}}
\def\img{{\rm Im\,}}
\def\Ker{{\rm Ker\,}}

\def \nind {\noindent}

\def \sskp {\smallskip}

\def \ms {\medskip}

\def \bs {\bigskip}

\begin{document}\selectlanguage{english}

\begin{frontmatter}



\title{On the local structure of noncommutative deformations}


\author{Mohamed Boucetta}
\ead{mboucetta2@yahoo.fr}

\author{Zouhair Saassai}
\ead{z.saassai@gmail.com}

\address{Université Cadi Ayyad, Faculté des Sciences et Techniques Gueliz, BP 549, 40000 Marrakech, Morocco}

\begin{abstract}
Let $(M,\pi,\mathcal{D})$ be a Poisson manifold endowed with a flat, torsion-free contravariant connection. We show that if $\mathcal{D}$
is an $\mathcal{F}$-connection then there exists a tensor $\mathbf{T}$ such that $\mathcal{D}\mathbf{T}$ is the metacurvature tensor introduced by E. Hawkins
in his work on noncommutative deformations. We compute $\mathbf{T}$ and the metacurvature tensor in this case, and show that if $\mathbf{T}=0$ then, near any regular
 point, $\pi$ and $\mathcal{D}$ are defined in a natural way by a Lie algebra action and a solution of the classical Yang-Baxter equation. Moreover,
when $\mathcal{D}$ is the contravariant Levi-Civita connection associated to $\pi$ and a Riemannian metric, the Lie algebra action preserves the metric.
\end{abstract}


\end{frontmatter}


\section{Introduction and main result}
\label{}
In \cite{haw1,haw2}, Hawkins showed that if a deformation of the graded algebra $\efd$ of
differential forms on a Riemannian manifold $M$ comes from a spectral triple describing $M$, then the Poisson
tensor $\pi$ (which characterizes the deformation) and the Riemannian metric satisfy
the following conditions:
\begin{enumerate}
  \item[$(H_1)$] The associated metric contravariante connexion $\mathcal{D}$ is flat.
  \item[$(H_2)$] The metacurvature of $\mathcal{D}$ vanishes.
  \item[$(H_3)$] The Poisson tensor $\pi$ is compatible with the Riemannian volume
$\mu$\,: $$d(i_{\pi}\mu)=0.$$
\end{enumerate}
The metric contravariant connection associated naturally to any pair of pseudo-Riemannian metric and Poisson tensor is the contravariant analogue of the classical Levi-Civita connection; it has appeared first in \cite{Bouc1}. The meta\-curvature, introduced in \cite{haw2}, is a (2,3)-tensor field (symmetric in the contravariant indices and antisymmetric in the covariant indices) associated naturally to any flat, torsion-free contravariant
connection.\\
The main result of Hawkins (cf. \cite{haw2}) states that if $(M,\pi,g)$ is a triple satisfying $(H_1)\!-\!(H_3)$ with $M$ 
compact, then around any regular point $x_0\in M$ the Poisson tensor can be written  as
\begin{equation}\label{haw thm}
\pi=\sum_{i,j}a_{ij}\,X_i\wedge X_j
\end{equation}
where $a_{ij}$ are constants and $\{X_1,\ldots,X_s\}$ is a family of linearly independent commuting Killing vector fields.
\par On the other hand, the first author showed in \cite{Bouc2} that if $\zeta:\g\to\ecv$ is an action of a finite-dimensional real Lie algebra $\g$ on a smooth manifold $M$, and $r\in\wedge^2\g$ is a solution of the classical Yang-Baxter equation, then the map $\D^r:\ef\times\ef\to\ef$ given by
\begin{equation}\label{Bouc contr connection}
\D^r_\alpha\beta:=\sum_{i,j=1}^n a_{ij}\,\alpha(\zeta(u_i))L_{\zeta(u_j)}\beta
\end{equation}
where $\{u_1,\ldots,u_n\}$ is any basis of $\g$ and $a_{ij}$ are the components of $r$ in this basis, depends only on $r$ 
and $\zeta$ and defines a flat, torsion-free contravariant connection with respect to the Poisson tensor $\pi^r:=\zeta(r)$. Moreover, if $M$ is Riemannian, then $\D^r$ is nothing else but the metric contravariant connection associated to the metric and $\pi^r$, provided that the action preserves the metric. He also showed that when $\g$ acts freely on $M$, the metacurvature of  $\D^r$ vanishes.\\
In this setting, \eqref{haw thm} can be reexpressed by saying that there exists a free action $\xi:\g\to\mathfrak{X}^1(U)$ 
of a finite-dimensional abelian Lie algebra $\g$ on an open neighborhood $U$ of $x_0$, which preserves $g$, and a solution 
$r\in\wedge^2\g$ of the classical Yang-Baxter equation such that $\pi=\pi^r$. Moreover, since $\xi$ preserves $g$, 
$\D=\D^r$ where $\D$ is the metric contravariant connection associated to $\pi$ and $g$. Therefore, $\D$ is a Poisson connection, i.e. $\D\pi=0$, and hence an $\F^{reg}$-connection (see \cite{Bouc3}).
\par Given a flat, torsion-free $\F^{reg}$-connection $\D$ on a Poisson manifold $(M,\pi)$, we shall see that there exists a (2,2)-type tensor field $\T$ on the dense open set of regular points such that
\begin{enumerate}
  \item[\rm (i)] $\D\T=\M\,$ where $\M$ is the metacurvature of $\D$;
  \item[\rm (ii)] $\T$ vanishes if and only if the exterior differential of any parallel 1-form is also parallel.
\end{enumerate}
By looking at the proof of Boucetta's result closely, one observes that in order to show that the metacurvature vanishes 
when the action is free, the first author shows, in fact, that $\D^r$ is an $\F^{reg}$-connection and that whenever a 
1-form is $\D^r$-parallel then so is its exterior differential, meaning that $\T$ vanishes (and hence so does the 
metacurvature). In the case studied by Hawkins $\T$ vanishes since as
we saw above the Lie algebra action $\xi$ is free.\\
So it is natural to study the following problem, inverse of Boucetta's result: Given a triple $(M,\pi,g)$ whose metric contravariant connection is a flat $\F^{reg}$-connection and such that ${\bf T}=0$, is there a free action of a finite-dimensional Lie algebra $\g$ preserving $g$ and a solution $r\in\wedge^2\g$ of the classical Yang-Baxter equation such that $\pi=\pi^r$ and $\D=\D^r$?

\ms The main result of this paper gives a positive answer to this question in a more general setting. More precisely,
\begin{thm}\label{main result} Let $(M,\pi,\D)$ be a Poisson manifold endowed with a flat, torsion-free contravariant connection.
\begin{enumerate}
 \item[\rm (1)] If $\,\D$ is an $\F^{reg}$-connection and ${\bf T}=0$, then
for any regular point $x_0$ with rank $2r$, there exists
a free action $\zeta:\mathfrak{g}\rightarrow\mathfrak{X}(U)$ of a $2r$-dimensional reel Lie algebra $\g$ on neighborhood $U$ of $x_0$, and an invertible solution $r\in\wedge^2\mathfrak{g}$ of the classical Yang-Baxter equation, such taht $\pi=\pi^r$ and $\D=\D^r$.
\item[\rm (2)] Moreover, if $\D$ is the metric contravariant connection associated to $\pi$ and a Riemannian metric $g$, then the action can be chosen in such a way that its fondamental vector fields are Killing.
    \end{enumerate}
\end{thm}

The paper is organized as follows. In  Section 2, we recall some standard facts about contravariant connections and the metacurvature tensor; we also define the tensor $\T$. Section 3 is devoted to the computation of the metacurvature tensor (and the tensor $\T$ as well) in the case of an $\F^{reg}$-connection. In the last section, we give a proof of Theorem \ref{main result}.

\begin{nota}
For a smooth manifold $M$, $\cinf(M)$ will denote the space of smooth functions on $M$, $\Gamma(V)$ will denote the space of smooth sections of a vector bundle $V$ over $M$, $\Omega^p(M):=\Gamma(\wedge^pT^*M)$ will denote the space of differential p-forms, and $\mathfrak{X}^p(M):=\Gamma(\wedge^pTM)$ will denote the space of p-vector fields.
\par For a Poisson tensor $\pi$ on $M$, we will denote by $\anch: T^*M\to TM$ the anchor map defined by $\beta(\anch(\alpha))=\pi(\alpha,\beta)$, and by $H_f$ the Hamiltonian vector field of a function $f$, that is, $H_f:=\anch(df)$. We well also denote by $[\;,\,]_\pi$ the Koszul-Schouten bracket on differential forms (see, e.g., \cite{Kosz}); this is given on 1-forms by $$[\alpha,\beta]_\pi=L_{\anch(\alpha)}\beta-L_{\anch(\beta)}\alpha-d\bigl(\pi(\alpha,\beta)\bigr).$$
The symplectic foliation of $(M,\pi)$ will be denoted by $\S$, and $T\S=\img\anch$ will be its associated tangent distribution. Finally, we will denote by $M^{reg}$ the dense open set where the rank of $\pi$ is locally constant.
\end{nota}
\section{Preliminaries}
\subsection{Contravariant connections}
Contravariant connections on Poisson manifolds were defined by Vaismann \cite{Vais} and studied in detail by Fernandes \cite{Fer}. These connections play an important role in Poisson geometry (see for instance
\cite{Fer, Crai-Marc}) and have recently turned out to be useful in other branches of mathematics (e.g., \cite{haw1,haw2}).

\sskp The definition of a contravariant connection mimics the usual definition of a covariant connection, except that cotangent vectors have taken the place of tangent vectors. More precisely, a {\em contravariant connection} on a Poisson manifold $(M,\pi)$ is an $\R$-bilinear map
\[\D:\ef\times\ef\to\ef,\;(\alpha,\beta)\longmapsto\D_\alpha\beta\] such that for any $f\in\efl$,
\[\D_{f\alpha}\beta=f\,\D_\alpha\beta\quad\mbox{and}\quad\D_\alpha(f\beta)=f\,\D_\alpha\beta+\anch(\alpha)(f)\beta.\]
A contravariant connection $\D$ is called an $\F$-{\em connection} \cite{Fer} if it satisfies
\[(\forall\,a\in T^*M,\;\anch(a)=0)\;\Longrightarrow\;\D_a=0.\]
We call $\D$ an $\F^{reg}$-{\em connection} if the restriction of $\D$ to $M^{reg}$ is an $\F$-connection.
\par The {\em torsion} and the {\em curvature} of a contravariant connection $\D$ are formally identical
to the usual ones:
\begin{align*}
T(\alpha,\beta) & =\D_{\alpha}\beta-\D_{\beta}\alpha-[\alpha,\beta]_{\pi}\,,\\
R(\alpha,\beta)\gamma & =\D_{\alpha}\D_{\beta}\gamma-\D_{\beta}\D_{\alpha}\gamma-\D_{[\alpha,\beta]_{\pi}}\gamma\,.
\end{align*}
These are (2,1) and (3,1)-type tensor fields, respectively. When $T\equiv0$ (resp. $R\equiv0$), $\D$ is called {\em torsion-free} (resp. {\em flat}).\\
In local coordinates $\coor$, the local components of the torsion and curvature tensor fields are given by
\begin{gather}
T_{ij}^k = \Gamma_{ij}^k-\Gamma_{ji}^k-\frac{\partial\pi_{ij}}{\partial x_k}\,,\label{torsion local exp}\\
R_{ijk}^l=\sum_{m=1}^d\Gamma_{im}^l\Gamma_{jk}^m-\Gamma_{jm}^l\Gamma_{ik}^m+
\pi_{im}\frac{\partial\Gamma_{jk}^l}{\partial x_m}-\pi_{jm}\frac{\partial\Gamma_{ik}^l}{\partial x_m}-\frac{\partial\pi_{ij}}{\partial x_m}\Gamma_{mk}^l\,,\label{curvature local exp}
\end{gather}
where $\Gamma_{ij}^k$ are the {\em Christoffel symbols} defined by: $\D_{dx_i}dx_j=\sum_{k=1}^d\Gamma_{ij}^k\,dx_k\,$.

\ms Given a (pseudo-)Riemannian metric $g$ on a Poisson manifold $(M,\pi)$, one has a contravariant version of the Levi-Civita connection: there exists a unique torsion-free contravariant connection $\D$ on $M$ which is metric-compatible, i.e.,
\[\pi_{\#}(\alpha)\!\cdot\!\langle\beta,\gamma\rangle=\langle\mathcal{D}_{\alpha}\beta,\gamma\rangle
+\langle\beta,\mathcal{D}_{\alpha}\gamma\rangle\quad\forall\,\alpha,\beta,\gamma\in\ef,\]
where $\prs$ denotes the metric pairing induced by $g$.
This connection is determined by the formula:
\begin{equation}\begin{split}\label{Koszul formula}
\langle\D_{\alpha}\beta,\gamma\rangle = \frac{1}{2}\bigl\{ & \anch(\alpha)\!\cdot\!\langle\beta,\gamma\rangle +\anch(\beta)\!\cdot\!\langle\alpha,\gamma\rangle -\anch(\gamma)\!\cdot\!\langle\alpha,\beta\rangle\\
& \!\!+\langle[\alpha,\beta]_{\pi},\gamma\rangle-\langle[\beta,\gamma]_{\pi},\alpha\rangle
+\langle[\gamma,\alpha]_{\pi},\beta\rangle\bigr\}\,,
\end{split}\end{equation}
and is called the {\em metric contravariant connection} (or {\em contravariant Levi-Civita connection}) associated to $(\pi,g)$.
\subsection{The metacurvature}
In this subsection we recall briefly from \cite{haw2} the definition of the meta\-curvature tensor and give some related formulas.
\par Let $(M,\pi)$ be a Poisson manifold. Given a torsion-free contravariant connection $\D$ on $M$, there exists a unique bracket $\cro$ on the space $\efd$ of differential forms, with the following properties:
\begin{enumerate}
\item $\cro$ is bilinear, degree 0 and antisymetric \begin{equation}\label{antisym}\{\sigma,\tau\}=-(-1)^{\deg(\sigma)\deg(\tau)}\{\tau,\sigma\}.\end{equation}
\item $\cro$ satisfies the product rule
    \begin{equation}\label{product rule}\{\sigma,\tau\wedge\rho\}=\{\sigma,\tau\}\wedge\rho
    +(-1)^{\deg(\sigma)\deg(\tau)}\tau\wedge\{\sigma,\rho\}.\end{equation}
\item The exterior differential $d$ is a derivation with respect to $\cro$, i.e.,
    \begin{equation}\label{d is derv}d\{\sigma,\tau\}=\{d\sigma,\tau\}
    +(-1)^{\deg(\sigma)}\{\sigma,d\tau\}.\end{equation}
\item For any $f,g\in\efl$ and any $\sigma\in\efl$, \begin{equation}\label{haw brackt on 00-01}
    \{f,g\}=\pi(df,dg)\quad\mbox{and}\quad\{f,\sigma\}=\mathcal{D}_{df}\sigma.\end{equation}
\end{enumerate}
This bracket is given (on decomposable forms) by
\begin{equation}\begin{split}
\{\alpha_1\wedge\cdots\wedge\alpha_k,\beta_1 & \wedge\cdots\wedge\beta_l\}=(-1)^{k+1}\sum_{i,j}
(-1)^{i+j}\{\alpha_i,\beta_j\}\wedge\\
&
\wedge\alpha_1\wedge\cdots\wedge\widehat{\alpha}
_i\wedge\cdots\wedge\alpha_k\wedge\beta_1\wedge\cdots\wedge
\widehat{\beta}_j\wedge\cdots\wedge\beta_l\,,
\end{split}\end{equation}
where the hat $\widehat{\,}$ denotes the absence of the corresponding factor, and the brackets $\{\alpha_i,\beta_j\}$ are given by the formula\footnote{This formula appeared first in \cite{bah-bouc}.}:
\begin{equation}\label{bah-bouc eq}
\{\alpha,\beta\}=-\mathcal{D}_{\alpha}d\beta-\mathcal{D}_{\beta}d\alpha+d\mathcal{D}_{\beta}
\alpha+[\alpha, d\beta]_{\pi}\,.
\end{equation}
We call the bracket $\cro$ {\em Hawkins bracket}.\\
Hawkins showed that the Hawkins bracket satisfies the graded Jacobi identity,
\begin{equation}\label{graded jacb ident}
\{\sigma,\{\tau,\rho\}\}-\{\{\sigma,\tau\},\rho\}-(-1)^{\deg(\sigma)\deg(\tau)}\{\tau,\{\sigma,\rho\}\}=0\,,
\end{equation}
if and only if $\D$ is flat and a certain 5-index tensor, called the metacurvature of $\D$, vanishes identically. In fact, Hawkins showed that if $\D$ is flat, then it determines a (2,3)-type tensor field $\M$ symmetric in the contravariant indices and antisymmetric in the covariant indices, given by
\begin{equation}\label{metacuvature def}
\M(df,\alpha,\beta)=\{f,\{\alpha,\beta\}\}-\{\{f,\alpha\},\beta\}-\{\alpha,\{f,\beta\}\}\,.
\end{equation}
The tensor $\M$ is the {\em metacurvature} of $\D$.

\ms The following formulas, due to Hawkins, will be useful later. Let $\alpha$ be a parallel 1-form; since $\D$ is torsion-free, $\,[\alpha,\eta]_\pi=\D_\alpha\eta\,$ for any $\eta\in\efd$, and so, by \eqref{bah-bouc eq}, the Hawkins bracket of $\alpha$ and any 1-form $\beta$ is given by
\begin{equation}\label{Haw bracket of 1-forms}
\{\alpha,\beta\}=-\D_{\beta}d\alpha\,.
\end{equation}
Using this, one can deduce easily from \eqref{metacuvature def} that for any parallel 1-forms $\alpha,\beta$ and any 1-form $\gamma$,
\begin{equation}\label{metacourbure formula}
\M(\gamma,\beta,\alpha)=-\D_\gamma\D_\beta d\alpha\,.
\end{equation}
\subsection{The tensor $\T$}
We now define the tensor $\T$, an essential ingredient in our main result.
\par Let $(M,\pi)$ be a Poisson manifold endowed with a flat, torsion-free, contravariant $\F^{reg}$-connection $\D$. For each $x\in M^{reg}$ and any $a,b\in T^*_xM$, define
\begin{equation}\label{tens T}
\T_x(a,b):=\{\alpha,\beta\}(x)\quad(\in\textstyle\bigwedge^2T_x^*M),
\end{equation}
where $\cro$ denotes the Hawkins bracket associated to $\D$, and $\alpha$ and $\beta$ are parallel 1-forms
defined in a neighborhood of $x$ such that $\alpha(x)=a$ and $\beta(x)=b$. (Such 1-forms exist, see Proposition \ref{flatness carct}.) This is independent of the choice of $\alpha$ and $\beta$ since by \eqref{Haw bracket of 1-forms} and \eqref{antisym} we have
\begin{equation}
\T_x(a,b)=-(\D_\alpha d\beta)(x)=-(\D_\beta d\alpha)(x).
\end{equation}
The assignment $x\mapsto\T_x$ is then a smooth (2,2)-type tensor field on $M^{reg}$, symmetric in the contravariant indices and antisymmetric in the covariant indices, which by \eqref{metacourbure formula} verifies $\D\T=\M$, and which clearly vanishes if and only if the exterior differential of any parallel 1-form is also parallel.

\section{Computation of the tensors $\M$ and $\T$}
The metacurvature tensor is rather difficult to compute in general. In the symplectic case, Hawkins has established a simple formula for the metacurvature \cite[Theorem 2.4]{haw2}. Bahayou and the first author have also established in \cite{bah-bouc} a formula for the metacurvature in the case of a Lie-Poisson group endowed with a left-invariant Riemannian metric. In this section we explain how to compute the metacuvature (and the tensor $\T$ as well), in the case of an $\F^{reg}$-connection, generalizing thus Hawkins's formula.
\par Throughout this section, $\D$ will be a torsion-free contravariant $\F^{reg}$-connection on a $d$-dimensional Poisson manifold $(M,\pi)$.

\bs We begin with the following simple lemma.
\begin{lem}\label{ker stability lem}
Let $U\subseteq M$ be an open set on which the rank of $\pi$ is constant. For any $\alpha,\beta\in\Omega^1(U)$,
$\anch(\beta)=0$ implies $\anch(\D_\alpha\beta)=0$, and in this case, $\,\D_\alpha\beta=\Li_{\anch(\alpha)}\beta$.
\end{lem}
In other words, the kernel of the anchor map restricted to $U$ is stable under $\D$. The next lemma shows that, around any regular point, there exists a complementary subbundle of $\Ker\anch$ which is also stable under $\D$, provided that $\D$ is flat.
\begin{lem}\label{flatness spliting lem}
If $\D$ is flat, then for any $x\in M^{reg}$ and any $\H_0\subseteq T_x^*M$ such that $\,T_x^*M=(\Ker\anch)_x\oplus\mathcal{H}_0$, the cotangent bundle splits smoothly around $x$ into: $$T^*M=(\Ker\anch)\oplus\mathcal{H}$$ with $\H$ stable under $\D$, i.e. $\D\H\subseteq\H$, and $\mathcal{H}_x=\mathcal{H}_0$.
\end{lem}
\begin{proof}
Let $(U;x_i,y_u)$ ($i=1,\ldots,2r$; $u=1,\ldots,d-2r$) be a local carte around $x$ such that
$$\pi=\frac{1}{2}\sum_{i,j=1}^{2r}\pi_{ij}\,\frac{\partial}{\partial x_i}\wedge\frac{\partial}{\partial x_j}$$
and the matrix $(\pi_{ij})_{1\leq i,j\leq 2r}$ is constant and invertible; let $(\pi^{ij})_{1\leq i,j\leq 2r}$ denote the inverse matrix. The restriction of $\Ker\anch$ to $U$ is a (rank $d-2r$) subbundle of $T_{|_U}^*M$, so we can choose a (arbitrary) smooth decomposition
$$T_{|_U}^*M=(\Ker\anch)\oplus\mathcal{H}\,.$$ Then clearly $\Ker\anch=\text{span}\{dy_u\}$, and
$$\mathcal{H}=\text{span}\left\{\theta_i=dx_i+\sum_{u=1}^{d-2r}B_i^u\,dy_u\right\}$$ for some functions $B_i^u\in\cinf(U)$. Since $\D$ is a torsion-free $\F$-connection on $U$, one has $\D_{dy_u}=\D dy_u=0$ for all $u$. Thus, for any $i,j$,
\begin{align*}
\D_{\theta_i}\theta_j & = \D_{dx_i}dx_j+\sum_{u=1}^{d-2r}\anch(dx_i)(B_j^u)\,dy_u\\
& = \left(\sum_{k=1}^{2r}\Gamma_{ij}^k\,dx_k+\sum_{u=1}^{d-2r}\Gamma_{ij}^u\,dy_u\right)+
\sum_{u=1}^{d-2r}\sum_{k=1}^{2r}\pi_{ik}\frac{\partial B_j^u}{\partial x_k}\,dy_u\\
& =
\sum_{k=1}^{2r}\Gamma_{ij}^k\,\theta_k+\sum_{u=1}^{d-2r}\left(\Gamma_{ij}^u+\sum_{k=1}^{2r}
\left(\pi_{ik}\frac{\partial B_j^u}{\partial x_k}-\Gamma_{ij}^kB_k^u\right)\right)dy_u\,,
\end{align*}
where $\Gamma_{ij}^k,\Gamma_{ij}^u$ are the Christoffel symbols of $\D$. Therefore, the desired decomposition exists if and only if we may find a family of local functions $\{B_i^u\}_{i,u}$ satisfying the following system of PDEs
\begin{equation*}
\Gamma_{ij}^u+\sum_{k=1}^{2r}\left(\pi_{ik}\frac{\partial B_j^u}{\partial
x_k}-\Gamma_{ij}^kB_k^u\right)=0\quad\forall\,i,j,\forall\,u\,,
\end{equation*}
or equivalently
\begin{equation*}
\frac{\partial B_j^u}{\partial
x_i}=\sum_{k=1}^{2r}\left(\sum_{l=1}^{2r}\pi^{il}\Gamma_{lj}^k\right)B_k^u-\sum_{l=1}^{2r}\pi^{il}\Gamma_{lj}^u
\quad\forall\,i,j,\forall\,u\,.\eqno{(*)}
\end{equation*}
In matrix notation, this is
\begin{equation*}
\frac{\partial}{\partial x_i}B^u=\Gamma_iB^u+Y_i^u\,,
\end{equation*}
where
$$B^u=\left(
        \begin{array}{c}
          B_1^u \\
          \vdots \\
          \vdots \\
          B_{2r}^u \\
        \end{array}
      \right);\;
\Gamma_i=\left(\sum_{m=1}^{2r}\pi^{im}\Gamma_{mk}^l\right)_{1\leq k,\,l\leq 2r};\;
Y_i^u=-\sum_{j=1}^{2r}\pi^{ij}\left(
        \begin{array}{c}
          \Gamma_{j1}^u \\
          \vdots\\
          \vdots\\
          \Gamma_{j\,2r}^u\\
        \end{array}
      \right).$$
Considering the $B_i^{u\,}$'s as functions with variables $x_i$ and parameters $y_{u\,}$, the system above can be solved, according to  Frobenius's Theorem (see, e.g., \cite[Theorem 1.1]{HakT}), if and only if the integrability conditions
\begin{equation*}
    \Gamma_i \Gamma_j+\frac{\partial}{\partial x_j}\Gamma_i=\Gamma_j
\Gamma_i+\frac{\partial}{\partial x_i}\Gamma_j\,,\quad
    \Gamma_i Y_j^u+\frac{\partial}{\partial x_j}Y_i^u=\Gamma_j Y_i^u+\frac{\partial}{\partial
x_i}Y_j^u\,,
\end{equation*}
hold for all $i,j$ and all $u$. With indices, these are respectively
\begin{gather*}
\sum_{m=1}^{2r}\Gamma_{im}^l\Gamma_{jk}^m-\Gamma_{jm}^l\Gamma_{ik}^m+\pi_{im}\frac{
\partial\Gamma_{jk}^l}{\partial x_m}-\pi_{jm}\frac{\partial\Gamma_{ik}^l}{\partial x_m}=0\,,\\
\sum_{m=1}^{2r}\Gamma_{im}^u\Gamma_{jk}^m-\Gamma_{jm}^u\Gamma_{ik}^m+\pi_{im}\frac{
\partial\Gamma_{jk}^u}{\partial x_m}-\pi_{jm}\frac{\partial\Gamma_{ik}^u}{\partial x_m}=0\,,
\end{gather*}
which by \eqref{curvature local exp} mean that the curvature vanishes. Thus $(*)$ has solutions (which depend smoothly on the parameters and the initial values).
\end{proof}

\begin{nota}
Given $\mathcal{H}$ as above, the restriction of $\anch$ to $\H$ defines an isomorphism from $\H$ onto $T\S$; we will denote by $\varpi^{\mathcal{H}}:T\S\rightarrow \mathcal{H}$ its inverse.
\end{nota}

\begin{prop}\label{flatness carct}
The following are equivalent:
\begin{enumerate}
   \item[\rm (a)] $\D$ is flat.
   \item[\rm (b)] For any $x\in M^{reg}$ and any $a\in T^*_x M$, there exists a 1-form
$\alpha$ defined in a neighborhood of $x$ such that $\alpha(x)=a$ and $\D\alpha=0$.
   \item[\rm (c)] Around any $x\in M^{reg}$, there exists a smooth coframe
$(\alpha_1,\ldots,\alpha_d)$ of $M$ such that $\D\alpha_i=0$ for all $i$. Such a coframe well be called flat.
\end{enumerate}
\end{prop}
\begin{proof}
The implications (b)$\;\Longrightarrow\;$(c) and (c)$\;\Longrightarrow\;$(a) are obvious. To show (a)$\;\Longrightarrow\;$(b), let $U\subseteq M$ be an open neighborhood of $x$ on which the rank of $\pi$ is constant. Over $U$, $T\S$ is a (involutive) regular distribution and $\D$ is a torsion-free $\F$-connection. So we can define a partial connection $\nabla$ on $T_{|_U}\S$ by setting for any $\alpha,\beta\in\Omega^1(U)$,
\begin{equation}\label{associated contr connection}
\nabla_{\anch(\alpha)}\anch(\beta)=\anch(\D_\alpha\beta).
\end{equation}
One verifies  immediately that the curvature tensor fields $R^\nabla$ and $R^\D$
respectively of $\nabla$ and $\D$  are related by:
$$R^\nabla\bigl(\anch(a),\anch(b)\bigr)\anch(c)=\anch\left(R^\D(a,b)c\right)\quad\forall\,a,b,
c\in T_{|_U}^*M,$$
and hence $R^\nabla$ vanishes since by hypothesis $R^\D$ does. Using  Frobenius's Theorem, we can then show in a way similar to the classical case that, for any $v\in T_x\S$, there exists a vector field $X$ defined on some neighborhood of $x$ such that $X(x)=v$, $X$ is tangent to $T\S$, that is,  $X(y)\in T_y\S$ for any $y$ near
$x$, and $\nabla X=0$.
\par Now let $a\in T_x^*M$. According to Lemma \ref{flatness spliting lem}, the cotangent bundle  splits smoothly around $x$ into: $T^*M=(\Ker\anch)\oplus\mathcal{H}$ with $\mathcal{H}$ stable under $\D$. $\mbox{Write}\;\,a=b+c$ with $b\in\Ker\anch(x)$ and $c\in\mathcal{H}_x$. By the argument above, there exists a $\nabla$-parallel vector field $X$ defined in a neighborhood of $x$ which is tangent to $T\S$ and such that $X(x)=\anch(c)$. Put
$\gamma=\varpi^{\mathcal{H}}(X)\in\Gamma(\mathcal{H})$; then $\gamma(x)=c$, and for any 1-form
$\phi$, $\anch(\D_\phi\gamma)=\nabla_{\anch(\phi)}X=0$ implying that $\D\gamma=0$. Taking $\,\alpha=\sum_{u=1}^{s}b_u\,dy_u+\gamma$, where $(y_u)$ is a family of local functions on $M$ such that $\Ker\anch=\text{span}\{dy_1,\ldots,dy_s\}$ near $x$, and $b_u$ are the coordinates of $b$ in $\{dy_1(x),\ldots,dy_s(x)\}$, we obtain finally the desired 1-form.
\end{proof}

The following corollary is a refinement of the preceding proposition.
\begin{cor}\label{flat coordn systm corol}
If $\D$ is flat, around any $x\in M^{reg}$ there exists an $\S$-foliated coordinate system with leafwise coordinates $\{x_i\}_{i=1}^{2r}$ and transverse coordinates $\{y_u\}_{u=1}^{d-2r}$  such that for any $\mathcal{H}$ as in Lemma \ref{flatness spliting lem},
$${\bf F}^*=\left(\phi_i:=\varpi^{\mathcal{H}}(\partial/\partial x_i)\,;\,dy_u\right)$$
is a flat coframe of $M$ near $x$. Such a coordinate system will be called flat.
\end{cor}

\begin{rmk}\label{flat coord systm rmk}
Another equivalent way of expressing that the $\S$-foliated coordinate system $(x_i,y_u)$ is flat is the following: $\nabla\partial/\partial x_i=0$ for all $i$, where $\nabla$ is the (local) partial connection defined by
\eqref{associated contr connection}.
\end{rmk}

\begin{center}
{\em We assume for the remainder of this section that $\D$ is flat}.
\end{center}

\sskp We shall compute the tensors $\M$ and $\T$ in the coframe ${\bf F^*}$. To do so, we need first to determine its dual frame.
\par With the notations of Corollary \ref{flat coordn systm corol}, for each $i$, there exist unique functions, $A_i^1,\ldots,A_i^{d-2r}$, defined in neighborhood of $x$ such that
\begin{equation}\label{structure functions}
dx_i+\sum_{u=1}^{d-2r}A_i^u\,dy_u\in\mathcal{H}\,.
\end{equation}
For any $i$ and any $u$ we put
\begin{equation}\label{repere dual plat}
X_i:=-H_{x_i}=-\anch(dx_i)\,,\quad Y_u:=\frac{\partial}{\partial
y_u}-\sum_{i=1}^{2r}A_i^u\,\frac{\partial}{\partial x_i}\,.
\end{equation}
\begin{lem}\label{dual frame lem}
With the above notations, $(X_i,Y_u)$ is the dual frame to ${\bf F}^*$.
Moreover, the vector fields $X_i$ and $Y_u$ are, respectively,
Hamiltonian and Poisson, and verify
\begin{equation}\begin{split}\label{dual frame-lie bracket}
& [X_i,X_j]= -\sum_{k=1}^{2r}\frac{\partial\pi_{ij}}{\partial x_k}\,X_k\,;\quad
[X_i,Y_u]= \sum_{j=1}^{2r}\frac{\partial A_i^u}{\partial x_j}\,X_j\,;\\
& [Y_u,Y_v]= \sum_{i,j=1}^{2r}\pi^{ij}\biggl(\frac{\partial A_j^u}{\partial y_v}-\frac{\partial
A_j^v}{\partial y_u}+\sum_{k=1}^{2r} A_k^u\frac{\partial A_j^v}{\partial
x_k}-A_k^v\frac{\partial A_j^u}{\partial x_k}\biggr)\,X_i\,.
\end{split}\end{equation}
Here, $\pi_{ij}:=\pi(dx_i,dx_j)$ and $(\pi^{ij})$ is the inverse matrix of $(\pi_{ij})$.
\end{lem}
\begin{proof}
The fact that $(X_i,Y_u)$ is the dual frame to ${\bf F^*}$ follows immediately, once we note that
\begin{equation}\label{phi expression}
\phi_i:=\varpi^{\mathcal{H}}(\partial/\partial x_i)=\sum_{j=1}^{2r}\pi^{ij}\Bigl(dx_j+\sum_{u=1}^{d-2r}A_j^u\,dy_u\Bigr).
\end{equation}
By definition, each of the vector fields $X_i$ is Hamiltonian.
To see that each $Y_u$ is Poisson, observe that the equation $[\phi_i,\phi_j]_\pi=0$ yields
\begin{align*}
Y_u\cdot\pi(\phi_i,\phi_j) & = L_{\partial/\partial
x_i}\phi_j\,(Y_u)-L_{\partial/\partial x_j}\phi_i\,(Y_u)\\
& = \textstyle-\phi_j\bigl(\bigl[\frac{\partial}{\partial x_i},Y_u\bigr]\bigr)
+\phi_i\bigl(\bigl[\frac{\partial}{\partial x_j},Y_u\bigr]\bigr)\\
& = \textstyle-L_{Y_u}\phi_j\bigl(\frac{\partial}{\partial x_i}\bigr)+Y_u\cdot\pi(\phi_i,\phi_j)
+L_{Y_u}\phi_i\bigl(\frac{\partial}{\partial x_j}\bigr)-Y_u\cdot\pi(\phi_j,\phi_i)\\
& = -\pi(\phi_i,L_{Y_u}\phi_j)-\pi(L_{Y_u}\phi_i,\phi_j)+2Y_u\cdot\pi(\phi_i,\phi_j),
\end{align*}
hence $\,L_{Y_u}\pi\,(\phi_i,\phi_j)=0\,$; in addition, we have
$$L_{Y_u}\pi\,(\phi_i,dy_v)=-\pi(\phi_i,L_{Y_u}dy_v)=-\pi(\phi_i,d(Y_u(y_v))=0\,,$$
and it is clear that we also have $L_{Y_u}\pi\,(dy_v,dy_w)=0$. It follows that $L_{Y_u}\pi=0$, which means that $Y_u$ is Poisson. Finally,
$$[X_i,X_j]=H_{\pi(dx_i,dx_j)}=-\sum_{k=1}^{2r}\frac{\partial\pi_{ij}}{\partial x_k}\,X_k\,,\quad
[X_i,Y_u]=H_{Y_u(x_i)}=\sum_{j=1}^{2r}\frac{\partial A_i^u}{\partial x_j}\,X_j\,,$$
and the last equality of \eqref{dual frame-lie bracket} follows by direct computation.
\end{proof}

We now can give the expression of the metacurvature in the coframe ${\bf F}^*$.
\begin{thm}\label{local exp for M thm}
With the same notations as above, we have
\begin{enumerate}
  \item[\rm (a)] For any $u=1,\ldots,d-2r$, $\,\M(dy_u,\cdot\,,\cdot\,)=0$.
  \item[\rm (b)] For any $i,j,k=1,\ldots,2r$,
      \begin{align*}
\M(\phi_i,\phi_j,\phi_k) & = -\sum_{l<m}\frac{\partial^3\,\pi_{lm}}{\partial x_i\partial
x_j\partial x_k}\,\phi_l\wedge\phi_m+\sum_{l,u}\frac{\partial^3\,A_l^u}{\partial x_i\partial
x_j\partial x_k}\,\phi_l\wedge dy_u\\
+\!\sum_{u<v,\,l}\frac{\partial^2}{\partial x_i\partial x_j} &
\biggl(\pi^{kl}\biggl(\frac{\partial A_l^u}{\partial y_v}-\frac{\partial A_l^v}{\partial
y_u}+\sum_m A_m^u\frac{\partial A_l^v}{\partial x_m}-A_m^v\frac{\partial A_l^u}{\partial
x_m}\biggr)\!\biggr)dy_u\wedge dy_v.
\end{align*}
\end{enumerate}
\end{thm}
\begin{proof}
Part (a) is immediate from \eqref{metacuvature def} and \eqref{haw brackt on 00-01}.
\par For (b), on the one hand, we have by \eqref{metacourbure formula},
$$\M(\phi_i,\phi_j,\phi_k)=-\D_{\phi_i}\D_{\phi_j}d\phi_k\quad\mbox{for\;all}\;i,j,k\,.$$
On the other hand, using Lemma \ref{dual frame lem} gives
\begin{equation}\begin{split}
d\phi_i= & \sum_{j<k}\frac{\partial\pi_{jk}}{\partial x_i}\,\phi_j\wedge\phi_k
-\sum_{j,u}\frac{\partial A_j^u}{\partial x_i}\,\phi_j\wedge dy_u\\
& -\!\sum_{u<v,\,j}\pi^{ij}\biggl(\frac{\partial A_j^u}{\partial y_v}-\frac{\partial
A_j^v}{\partial y_u}+\sum_k A_k^u\frac{\partial A_j^v}{\partial x_k}-A_k^v\frac{\partial
A_j^u}{\partial x_k}\biggr)dy_u\wedge dy_v\,,
\end{split}\end{equation}
and the desired formula follows.
\end{proof}
Likewise, we get the following expression for the tensor $\T$.
\begin{thm}\label{loc exp for T thm}
\begin{enumerate}
\item[\rm (i)] For any $u=1,\ldots,d-2r$, $\,\T(dy_u,\cdot\,)=0$.
\item[\rm (ii)] For any $i,j,k=1,\ldots,2r$,
\begin{align*}
\T(\phi_i,\phi_j)= & -\sum_{k<l}\frac{\partial^2\,\pi_{kl}}{\partial x_i\partial x_j}
\,\phi_k\wedge\phi_l+\sum_{k,u}\frac{\partial^2\,A_k^u}{\partial x_i\partial
x_j}\,\phi_k\wedge dy_u\\
+\!\sum_{u<v,\,k}\frac{\partial}{\partial x_i} &
\biggl(\pi^{jk}\biggl(\frac{\partial A_k^u}{\partial y_v}-\frac{\partial
A_k^v}{\partial y_u}+\sum_l A_l^u\frac{\partial A_k^v}{\partial x_l}-A_l^v\frac{\partial A_k^u}
{\partial x_l}\biggr)\!\biggr)dy_u\wedge dy_v.
\end{align*}
\end{enumerate}
\end{thm}

\subsection{The symplectic case}
If the Poisson tensor $\pi$ is invertible, then the flat and torsion-free contravariant connection $\D$ is an $\F$-connection\footnote{Actually, this is true for any contravariant connection on $M$ since $\Ker\anch=\{0\}$.}, and is related to a flat, torsion-free, covariant connection $\nabla$ on $M$ via $\anch(\D_\alpha\beta)=\nabla_{\anch(\alpha)}\anch(\beta)$. In that case, a flat coordinate system is one with respect to whom $\nabla$ is given trivially by partial derivatives (Remark \ref{flat coord systm rmk}).
\par Since the kernel of the anchor map reduces to zero, the metacurvature vanishes if and only if $\pi$ is quadratic in the affine structure defined by $\nabla$ (Theorem \ref{local exp for M thm}), which is precisely the conclusion of \cite[Theorem 2.4]{haw2}.
\par Likewise, the tensor $\T$ vanishes if and only if the components of $\pi$ w.r.t. any flat coordinate system are at most of degree one (Theorem \ref{loc exp for T thm}).
\begin{exple}
If $\D$ is a flat, torsion-free, Poisson connection on a Poisson manifold $(M,\pi)$ with $\pi$ invertible, then $\T$ vanishes identically. In fact, the condition  $\D\pi=0$ is equivalent to saying that the components of $\pi$ with respect to any flat coordinate system are constant.
\end{exple}
\begin{exple}
Let $G$ be a Lie group with Lie algebra $\g$, and let $r\in\wedge^2\g$ be a solution of the classical Yang-Baxter equation. For any tensor $\tau$ on $\g$, denote by $\tau^+$ the corresponding left-invariant tensor field on $G$. Following \cite{Bouc2}, the formula $$\D^r_{a^+}\,\!b^+=-(\ad^*_{r(a)}b)^+,$$ where $a,b\in\g^*$, defines a left-invariant, flat, torsion-free, $\F$-connection $\D^r$ on $(G,r^+)$ with vanishing $\T$. It is well known (see, e.g., \cite{Chu}) that if $r$ is invertible, then the left-invariant symplectic form $\omega^+$ inverse of $r^+$ defines a left-invariant, flat, torsion-free connection $\nabla$ on $G$ via $$\omega^+(\nabla_{u^+}v^+,w^+)=-\omega^+(v^+,[u^+,w^+]),\quad u,v,w\in\g\,.$$
One then checks easily that $\D^r$ and $\nabla$ are related by: $\!r^+_\#(\D^r_{a^+}\,\!b^+)\!=\!\nabla_{r(a)^+}r(b)^+\!$. We thus recover a result of the first author and Medina (cf. \cite[Theorem 1.1-(1)]{Bouc-Med}) which states that if $r$ is invertible, then $r^+$ is polynomial of degree at most 1 with respect to the affine structure defined by $\nabla$.
\end{exple}

\subsection{The Riemannian case}
Let $\D$ be the metric contravariant connection associated to a Poisson tensor $\pi$ and a Riemannian metric $g$ on a manifold $M$. Thanks to the metric $g$, the cotangent bundle splits orthogonally into
$$T^*M=\Ker\anch\oplus(\Ker\anch)^{\perp}.$$
\begin{lem}\label{CLCC orth ker}
Let $U\subseteq M$ be an open set on which the rank of $\pi$ is constant. Assume that $\D$ is an $\F$-connection on $U$. Then $(\Ker\anch_{|_U})^\perp$ is stable $\mbox{under}\;\,\D$.
\end{lem}
Thus if $\D$ is flat and an $\F^{reg}$-connection, then by Corollary \ref{flat coordn systm corol} there exists around any $x\in M^{reg}$ an $\S$-foliated chart with leafwise coordinates $\{x_i\}_{i=1}^{2r}$ and transverse coordinates $\{y_u\}_{u=1}^{d-2r}$ such that $\left\{\phi_i:=\varpi^{\perp}(\partial/\partial x_i)\,;\,dy_u\right\}$
is a flat coframe of $M$ near $x$, where we have denoted by $\varpi^{\perp}:T\S\rightarrow(\Ker\anch)^{\perp}$ the inverse of $\anch:(\Ker\anch)^{\perp}\rightarrow T\S$. In this case, the functions $A_i^u$ defined by \eqref{structure functions} can be computed by means of the metric; indeed, using \eqref{phi expression} and the fact that $\langle\phi_i,dy_u\rangle=0$, one has $-A_i^u=\sum_v g_{iv}g^{uv}$ where $g_{iv}=\langle dx_i,dy_v\rangle$ and $(g^{uv})$ is the inverse matrix of the one whose coefficients are $g_{uv}=\langle dy_u,dy_v\rangle$.
\section{Proof of Theorem \ref{main result}}
Let $(x_i,y_u)$, with $i=1,\ldots,2r$ and $u=1,\ldots,d-2r$, be a flat coordinate system around $x_0$, choose $\mathcal{H}$ as in lemma \ref{flatness spliting lem}, and let ${\bf F}^*=\{\phi_i,dy_u\}$ be the corresponding flat coframe and $\{X_i,Y_u\}$ its dual frame. We shall construct a family of vector fields $\{Z_1,\ldots,Z_{2r}\}$ on a neighborhood $U$ of $x_0$ which span $T\S$ and commute with the $X_i$'s and the $Y_u$'s. In that case,
\begin{itemize}
  \item The family $\{Z_1,\ldots,Z_{2r}\}$ will form a 2r-dimensional reel Lie algebra $\g$, since by the Jacobi identity $$[[Z_i,Z_j],X_l]=[[Z_i,Z_j],Y_u]=0\quad\forall\,i,j,l,\,\forall\,u,$$ so that $[Z_i,Z_j]=\sum_{k}c_{ij}^k\,Z_k$ with $c_{ij}^k$ being constant; it is then clear that $\g$ acts freely on $U$.
  \item The Poisson tensor $\pi$ will be expressed as $$\pi=\frac{1}{2}\,\sum_{i,j}
      a_{ij}\,Z_i\wedge Z_j$$ where the matrix $(a_{ij})_{1\leq i,j\leq 2r}$ is constant and invertible: since the  $X_i$'s and the $Y_u$'s are Poisson (Lemma \ref{dual frame lem}), then writing $\,\pi=\sum_{i<j}a_{ij}\,Z_i\wedge Z_j$ where $a_{ij}\in\cinf(U)$, we get $X_k(a_{ij})=Y_u(a_{ij})=0$.
  \item The connection $\D$ will be given on $U$ by $$\D_\alpha\beta=\sum_{i,j}a_{ij}\alpha(Z_i)\,\Li_{Z_j}\beta\,.$$ In fact, this is true for any $\beta\in{\bf F}^*$ since $\,\Li_{Z_i}\phi_j=\Li_{Z_i}dy_u=0\,$, and $\D_\alpha\beta-\sum_{i,j}a_{ij}\alpha(Z_i)\,\Li_{Z_j}\beta$ is tensorial in $\beta$ as $\anch(\alpha)=\sum_{i,j}a_{ij}\alpha(Z_i)Z_j$.
\end{itemize}

\sskp\nind We shall proceed in two steps. We first construct a family of vector fields which span $T\S$ and commute with the $X_i$'s, and then construct from this the desired family.\\
To start, observe that by virtue of Theorem \ref{} and Lemme \ref{dual frame lem} we have
$$[X_i,X_j]=\sum_{k=1}^{2r}\lambda_{ij}^k\,X_k\,,\quad[X_i,Y_u]=\sum_{j=1}^{2r}\mu_{iu}^j\,X_j
\,,\quad[Y_u,Y_v]=\sum_{i=1}^{2r}\nu_{uv}^i\,X_i\,,$$
where $\lambda_{ij}^k,\,\mu_{iu}^j,\,\nu_{uv}^i$ are Casimir functions. Let $\mathcal{T}\subseteq M$ be a smooth transversal to $T\S$ intersecting $x_0$; this is parametrized by the $y_u$'s. Fixing $y\in \mathcal{T}$,
the restrictions $X_1^y,\ldots,X_{2r}^y$ of $X_1,\ldots,X_{2r}$ to the symplectic leaf $\S_y$ passing through $y$ form a Lie algebra $\g_y$ which acts freely and transitively on $\S_y$. Therefore, according to \cite{Michor}, there exists a free transitive Lie algebra anti-homomorphism $\hat{\Gamma}_{y}:\g_y\rightarrow\mathfrak{X}^1(\S_y)$ whose image is $$\hat{\Gamma}_y(\g_y)=\bigl\{T\in\mathfrak{X}^1(\S_y):\,[T,X_i^y]=0\;\;\forall\,i=1,\ldots,2r\bigr\},$$ and such that $\hat{\Gamma}_y(X_i^y)(y)=X_i(y)$ for all $i$. Setting for any $i$,
$$T_i(z):=\hat{\Gamma}_y(X_i^y)(z),\quad z\in\S_y$$
and varying $y$ along $\mathcal{T}$, we get a family of linearly independent vector fields $\{T_1,\ldots,T_{2r}\}$ which are tangent to $T\S$ and verify
$$[T_i,X_j]=0\quad\mbox{for\;all}\;i,j,$$
and such that $T_i(y)=X_i(y)$ for all $i$ and all $y\in \mathcal{T}$. Note that $T_1,\ldots,T_{2r}$ are smooth since the solutions of the system $$[T,X_i]=0\,,\quad\,i=1,\ldots,2r$$
depend smoothly on the parameter $y\in \mathcal{T}$ and the initial values along $\mathcal{T}$. It is also worth noting that since the $\mu_{iu}^j$'s are Casimir, we have
$$[X_i,[T_j,Y_u]]=0\quad\mbox{for\;all}\;i,j\;\mbox{and\;all}\;u\,,$$
so that
$$[T_i,Y_u]=\sum_{j=1}^{2r}\gamma_{iu}^j\,T_j\,,$$ where $\gamma_{iu}^j$ are Casimir functions; in addition, since the $\nu_{uv}^i$'s are Casimir, we have
$$[T_i,[Y_u,Y_v]]=0\quad\mbox{for\;all}\;i\;\mbox{and\;all}\;u,v$$
implying
$$\frac{\partial\gamma_{ju}^i}{\partial y_v}-\frac{\partial\gamma_{jv}^i}{\partial
y_u}+\sum_{k=1}^{2r}\gamma_{ku}^i\gamma_{jv}^k-\gamma_{kv}^i\gamma_{ju}^k=0\eqno{(*)}$$
for all $i,j$ and all $u,v$.

\noindent Now we would like to find an invertible matrix $\xi=(\xi_{ij})_{1\leq i,j\leq 2r}$ where $\xi_{ij}$ are Casimir functions such that the vector fields
$$Z_i:=\sum_{j=1}^{2r}\xi_{ji}\,T_j\,,\quad i=1,\ldots,2r$$
verify $$[Z_i,Y_u]=0\quad\mbox{for\;all}\;i\;\mbox{and\;all}\;u\,.$$
If such a matrix exists, the family $\{Z_1,\ldots,Z_{2r}\}$ is clearly the desired one. Since the functions $\xi_{ij}$ are searched to be Casimir, the condition for the $Z_i$'s to commute with the $Y_u$'s can be rewritten as $$\frac{\partial\xi_{ji}}{\partial
y_u}=\sum_{k=1}^{2r}\gamma_{ku}^j\,\xi_{ki}\quad\forall\,i,j,\forall\,u\,,$$
or in matrix notation
$$\frac{\partial}{\partial y_u}\,\xi_{i}=\Gamma_u\,\xi_{i}\,,$$
where $\xi_1,\ldots,\xi_{2r}$ are the colon row of $\xi$ and $\Gamma_u:=(\gamma_{ju}^i)_{1\leq i,j\leq 2r}$. So we need to solve this system. Again, since the functions $\xi_{ij}$ are searched to be Casimir, we can solve it on $\mathcal{T}$. According to  Frobenius's Theorem, this system has solutions if and only if the following integrability condition
\begin{equation*}
\Gamma_u \Gamma_v+\frac{\partial}{\partial y_v}\Gamma_u=\Gamma_v
\Gamma_u+\frac{\partial}{\partial y_u}\Gamma_v
\end{equation*}
holds for all $u,v$, which is nothing else but $(*)$. It then suffices to take $\xi_{ij}(x_0)=\delta_{ij}$ as initial conditions to conclude. \\
Finally, if $\D$ is the metric contravariant connection with respect to $\pi$ and a Riemannian metric $g$, we choose $\mathcal{H}=(\Ker\anch)^\perp$. In this case, we have
$$\Li_{Z_i}g\,(\phi_j,\phi_k)=\Li_{Z_i}g\,(\phi_j,dy_u)=\Li_{Z_i}g\,(dy_u,dy_v)=0$$
since $\,\Li_{Z_i}\phi_j=\Li_{Z_i}dy_u=0$ and since $\,g(\phi_i,\phi_j)$ and $g(dy_u,dy_v)$ are Casimir functions. This shows that the vector fields $Z_i$ are Killing. \hfill{}$\Box$





\bibliographystyle{model1a-num-names}
\bibliography{<your-bib-database>}

\begin{thebibliography}{00}

\bibitem{bah-bouc}  A. Bahayou, M. Boucetta, \newblock{\em Metacurvature of Riemannian Poisson-Lie groups}, \newblock Journal of Lie Theory, Vol. 19 (2009) 439-462.
\bibitem{Chu} Chu, Bon-Yao, \newblock{\em Symplectic homogeneous spaces}, \newblock Transactions of the AMS, 197 (1974), 145-159.
\bibitem{Michor} D. V. Alekseevsk, P. W. Michor, \newblock{\em Differential geometry of $\g$-manifolds}, \newblock Differantial Geometry and its Applications, 5 (1995) 371-403 North-Holland.
\bibitem{haw1} E. Hawkins, \newblock{\em Noncommutative rigidity}, \newblock Commun. Math. Phys.  246 (2004) 211-235.
\bibitem{haw2} $\overline{\qquad\quad}$, \newblock{\em The structure of noncommutative deformations}, \newblock J. Diff. Geom. 77 (2007) 385-424.
\bibitem{HakT} H. A. Hakopian, M. G. Tonoyan, \newblock{\em Partial differential analogs of ordinary differential equations and systems}, \newblock New York J. Math. 10 (2004) 89-116.
\bibitem{Vais} I. Vaismann, \newblock{\em Lectures on the geometry of Poisson manifolds}, \newblock Progr. in Math. Vol. 118, Birkhäsher, Berlin 1994.
\bibitem{Kosz} J.-L. Koszul, \newblock{\em Crochet de Schouten-Nijenhuis et cohomologie}. \newblock In: Elie Cartan et les mathématiques d'aujourd'hui, Astérisque hors série, (1985) 257-271.
\bibitem{Bouc1} M. Boucetta, \newblock{\em Compatibilités des structures pseudo-riemanniennes et des structrues de Poisson}, \newblock C. R. Acad. Sci. Paris Sér. I 333 (2001) 763-768.
\bibitem{Bouc3} $\overline{\qquad\quad}$, \newblock{\em Poisson manifolds with compatible pseudo-metric and pseudo-Riemannian Lie algebras}, \newblock Differential Geometry and its Applications, Vol. 20, Issue 3(2004), 279-291.
\bibitem{Bouc2} $\overline{\qquad\quad}$, \newblock{\em Solutions of the classical Yang-Baxter equation and non-commutative deformations}, \newblock Letters in Mathematical Physics (2008) 83:69-81.
\bibitem{Bouc-Med} M. Boucetta, A. Medina, \newblock{\em Polynomial Poisson structures on affine solvmanifolds}, \newblock J. Symplectic Geom. Vol. 9, Number 3 (2011), 387-401.
\bibitem{Crai-Marc} M. Crainic, I. Marcut, \newblock{\em On the existence of symplectic realizations}, \newblock J. Symplectic Geom. Vol. 9, Number 4 (2011), 435-444.
\bibitem{Fer} R. L. Fernandes, \newblock{\em Connections in Poisson Geometry I: holonomy and invariants}, \newblock J. Diff. Geom. 54 (2000) 303-366.
\end{thebibliography}



\end{document}